\documentclass[12pt,reqno]{amsart}

\usepackage[usenames, dvipsnames, svgnames]{xcolor}
\usepackage{amsmath, amssymb,graphicx,amsthm,latexsym, amsfonts, enumitem, mathtools, tensor, MnSymbol}
\usepackage{hyperref}
\usepackage[all, color]{xy}
\usepackage{color}
\usepackage{amssymb}
\usepackage{float}
\usepackage{tikz}
\usetikzlibrary{arrows,decorations.pathmorphing,decorations.pathreplacing,positioning,shapes.geometric,shapes.misc,decorations.markings,decorations.fractals,calc,patterns}

\DeclareMathOperator{\Hom}{Hom}

\DeclareMathOperator{\add}{add}

\theoremstyle{plain}
\newtheorem{theorem}{Theorem}[section]
\newtheorem*{theorem*}{Theorem}

\theoremstyle{definition}
\newtheorem{defn}[theorem]{Definition}

\newtheorem{exmp}[theorem]{Example} 
\newtheorem{remark}[theorem]{Remark}

\newtheorem{lemma}[theorem]{Lemma}
\newtheorem{notation}[theorem]{Notation}
\newtheorem{corollary}[theorem]{Corollary}

\newtheorem{setup}[theorem]{Setup}

\date{}
\begin{document}
\setlength{\parindent}{0pt}
\setlength{\parskip}{7pt}
\address{School of Mathematics, Statistics and Physics,
Newcastle University, Newcastle upon Tyne NE1 7RU, United Kingdom}
\email{francesca.fedele@gmail.com}
\keywords{Hom-spaces, limits and colimits, $(n-1)$-Calabi-Yau, $(n-1)$-cluster tilting, $n$-Calabi-Yau triple, quotient categories,  truncation triangles.}
\subjclass[2010]{16E45, 18A30, 18E30.}
\title[Triangulated and quotient categories arising from $n$-Calabi-Yau triples]{Properties of triangulated and quotient categories arising from $n$-Calabi-Yau triples}
\author{Francesca Fedele}
%\address{School of Mathematics, Statistics and Physics,
%Newcastle University, Newcastle upon Tyne NE1 7RU, United Kingdom}
%\email{F.Fedele2@newcastle.ac.uk}
%\keywords{Grothendieck group, higher-angulated category, index, $m$-cluster tilting subcategory.}
%\subjclass[2010]{16E20, 18E30.}
\maketitle
\begin{abstract}
The original definition of cluster algebras by Fomin and Zelevinsky has been categorified and generalised in several ways over the course of the past 20 years, giving rise to cluster theory. This study lead to Iyama and Yang's generalised cluster categories $\mathcal{T}/\mathcal{T}^{fd}$ coming from $n$-Calabi-Yau triples $(\mathcal{T}, \mathcal{T}^{fd}, \mathcal{M})$. In this paper, we use some classic tools of homological algebra to give a deeper understanding of such categories $\mathcal{T}/\mathcal{T}^{fd}$.
    
Let $k$ be a field, $n\geq 3$ an integer and $\mathcal{T}$ a $k$-linear triangulated category with a triangulated subcategory $\mathcal{T}^{fd}$ and a subcategory $\mathcal{M}=\text{add}(M)$ such that $(\mathcal{T}, \mathcal{T}^{fd}, \mathcal{M})$ is an $n$-Calabi-Yau triple.
    %Then, the category $\mathcal{T}/\mathcal{T}^{fd}$ is a generalised cluster category.
    For every integer $m$ and every object $X$ in $\mathcal{T}$, there is a unique, up to isomorphism, truncation triangle of the form
    \begin{align*}
        X^{\leq m}\rightarrow X\rightarrow X^{\geq m+1}\rightarrow\Sigma X^{\leq m},
    \end{align*}
    with respect to the $t$-structure $((\Sigma^{<-m}\mathcal{M})^{\perp_\mathcal{T}},(\Sigma^{>-m}\mathcal{M})^{\perp_\mathcal{T}})$.
    In this paper, we prove some properties of the triangulated categories $\mathcal{T}$ and $\mathcal{T}/\mathcal{T}^{fd}$. Our first result gives a relation between the Hom-spaces in these categories, using limits and colimits. Our second result is a Gap Theorem in $\mathcal{T}$, showing when the truncation triangles split.
    
    Moreover, we apply our two theorems to present an alternative proof to a result by Guo, originally stated in a more specific setup of dg $k$-algebras $A$ and subcategories of the derived category of dg $A$-modules. This proves that $\mathcal{T}/\mathcal{T}^{fd}$ is Hom-finite and $(n-1)$-Calabi-Yau, its object $M$ is $(n-1)$-cluster tilting and the endomorphism algebras of $M$ over $\mathcal{T}$ and over $\mathcal{T}/\mathcal{T}^{fd}$ are isomorphic. Note that these properties make $\mathcal{T}/\mathcal{T}^{fd}$ a generalisation of the cluster category.
\end{abstract}
\section{Introduction}

Cluster theory has its origins with the definition of cluster algebras by Fomin and Zelevinsky in \cite[Definition 2.3]{FZ} . This was later categorified in several ways, starting with the definition of cluster categories by Buan, Marsh, Reineke,
Reiten, and Todorov in \cite[Section 1]{BRMRT}. A widely used generalisation of cluster categories is based on the Ginzburg dg algebra, see \cite[Section 4.2]{Gin}.
Starting from this, Amiot in \cite{AC} and Guo in \cite{GL} gave a better and more general way to construct cluster categories. In fact, unlike the original one, their definition  does not rely on orbit categories but on the more general concept of Verdier quotients.

In \cite{IYa}, Iyama and Yang generalised the above even further. In fact, instead of coming from a dg algebra, their generalised cluster categories come from an $n$-Calabi-Yau triple $(\mathcal{T},\mathcal{T}^{fd},\mathcal{M})$ in the sense of Definition \ref{defn_nCYtriple} with $\mathcal{M}=\text{add}\,(M)$.
The idea behind such a triple is that, starting from a triangulated category $\mathcal{T}$ with some extra assumptions and a triangulated subcategory $\mathcal{T}^{fd}$, we obtain a triangulated quotient category $\mathcal{T}/\mathcal{T}^{fd}$ playing the role of a generalised cluster category, see \cite[Section 2]{GL} and \cite[Section 5.3]{IYa}. Note that an $n$-Calabi-Yau triple is not just a vehicle to construct a generalised cluster category, but it is also a vehicle using $\mathcal{T}$ as a ``model'' to compute in $\mathcal{T}/\mathcal{T}^{fd}$.

Iyama and Yang proved in \cite[Section 5]{IYa} that such a generalised  cluster category has the requisite properties, that is it is Hom-finite, $(n-1)$-Calabi-Yau and it has an $(n-1)$-cluster tilting object. The aim of this paper is to reprove these properties using completely different and more classic means, giving in this way a deeper understanding of the generalised cluster category $\mathcal{T}/\mathcal{T}^{fd}$. In our approach, we use two classic tools in homological algebra: we use limits and colimits to compute Hom spaces in the quotient  $\mathcal{T}/\mathcal{T}^{fd}$ and we present a Gap theorem.

%OLD NOW CHANGED
%Let $k$ be a field, $n\geq 3$ be an integer and $(\mathcal{T},\mathcal{T}^{fd},\mathcal{M})$ be an $n$-Calabi-Yau triple in the sense of Definition \ref{defn_nCYtriple}, with $\mathcal{M}=\text{add}\,(M)$ for some object $M$ in $\mathcal{T}$.
%Note that the Ginzburg dg algebra, see \cite[Section 4.2]{Gin}, permits to define a generalisation of cluster categories and $n$-Calabi-Yau triples are a further generalisation of this.  In fact, the idea behind such a triple is that, starting from a triangulated category $\mathcal{T}$ with some extra assumptions and a triangulated subcategory $\mathcal{T}^{fd}$, we obtain a triangulated quotient category $\mathcal{T}/\mathcal{T}^{fd}$ which is a generalised cluster category, see \cite[Section 2]{GL} and \cite[Section 5.3]{IYa}. Note that an $n$-Calabi-Yau triple is not just a vehicle to construct a generalised cluster category, but it is also a vehicle using $\mathcal{T}$ as a ``model'' to compute in $\mathcal{T}/\mathcal{T}^{fd}$.
%In this paper, we prove some new results about $\mathcal{T}$ and their relation to $\mathcal{T}/\mathcal{T}^{fd}$ and we use them to give a new alternative proof that $\mathcal{T}/\mathcal{T}^{fd}$ is a generalised cluster category. 

For every integer $m$, the pair
\begin{align*}
    (\mathcal{T}^{\leq m},\mathcal{T}^{\geq m}):=((\Sigma^{<-m}\mathcal{M})^{\perp_\mathcal{T}}, (\Sigma^{>-m}\mathcal{M})^{\perp_\mathcal{T}})
\end{align*}
is a $t$-structure. Hence, for every object $X$ in $\mathcal{T}$, there is a unique (up to isomorphism) truncation triangle 
\begin{align*}
    X^{\leq m}\xrightarrow{f^m} X\xrightarrow{g^{m+1}} X^{\geq m+1}\rightarrow\Sigma X^{\leq m},
\end{align*}
with $X^{\leq m}\in\mathcal{T}^{\leq m}$ and $X^{\geq m +1}\in\mathcal{T}^{\geq m+1}$.

An example of the above setup, which has been studied by Amiot in \cite{AC} and Guo in \cite{GL} among others, is given when $\mathcal{T}$ and $\mathcal{T}^{fd}$ are certain subcategories of the derived category of dg $A$-modules, where $A$ is a dg $k$-algebra and $M=A$, see Remark \ref{remark_AG_setup}. In particular, considering a quiver with potential and its Ginzburg dg algebra $\Gamma$, if $H^0(\Gamma)$ is finite-dimensional, then we get a $3$-Calabi-Yau triple, see Example \ref{example_3cy}.

%In this paper, we study the triangulated category $\mathcal{T}$ and the triangulated quotient category $\mathcal{T}/\mathcal{T}^{fd}$ and prove some of their properties.
Our first theorem compares Hom-spaces in $\mathcal{T}$ and Hom-spaces in $\mathcal{T}/\mathcal{T}^{fd}$ and gives a relation between the two using limits and colimits. Note that, despite the different setups, our result recalls some results by Artin and Zhang. The fact that the direct system in (a) stabilizes resembles \cite[Proposition 3.13]{AZ} and the formula in (a) resembles (2.2.1) from \cite[proof of Proposition 2.2]{AZ}.

{\bf Theorem A (=Theorem \ref{thm_directstab}).}
{\em Let $X$ and $Y$ be objects in $\mathcal{T}$.
\begin{enumerate}[label=(\alph*)]
\item For $p\gg0$, the direct system
\begin{align*}
\xymatrix@C=3em{
    \mathcal{T}(X,Y)\ar[r]&\mathcal{T}(X^{\leq 0},Y)\ar[r]&\mathcal{T}(X^{\leq -1},Y)\ar[r]&\mathcal{T}(X^{\leq -2},Y)\ar[r]&\cdots
    }
\end{align*}
stabilizes. Moreover, we have that
    $\underset{q}{\varinjlim}\,\mathcal{T}(X^{\leq -q}, Y)\cong \mathcal{T}/\mathcal{T}^{fd}(X, Y)$.
\item For $X$ and $Y$ in $\mathcal{T}$, we have that
    \begin{align*}
       \underset{q}{\varprojlim} \Big(\underset{p}{\varinjlim}\,\mathcal{T}(\Sigma^{-1}X^{\geq -p+1},Y^{\leq -q})\Big)\cong \mathcal{T}/\mathcal{T}^{fd}(X,Y).
    \end{align*}
\end{enumerate}
}

Our second theorem is inspired by the Gap Theorem by Frankild and J{\o}rgensen, see \cite[Theorem 2.5]{FJ}. The vanishing condition on the Hom-spaces can be viewed as a gap in the cohomology of the object $X$ and it implies the splitting of truncation triangles.

{\bf Theorem B(= Theorem \ref{thm_gapinX}).}
{\em Let $a$ be an integer and $X\in\mathcal{T}$ be such that $\mathcal{T}(M,\Sigma^j X)=0$ for $a\leq j\leq a+n-2$. Then, the truncation triangle
\begin{align*}
    \xymatrix@C=3em{
    X^{\leq a-1}\ar[r]& X\ar[r]&X^{\geq a}\ar[r]^-\delta& \Sigma X^{\leq a-1}
    }
\end{align*}
splits and $X\cong X^{\leq a-1}\oplus X^{\geq a}$.
}

We apply our two theorems to give an alternative proof to the following result. This was first proven, in the more specific setup of dg $k$-algebras and their derived categories outlined above, by Amiot in \cite[Theorem 2.1]{AC} for the case $n=3$ and by Guo in \cite[Theorem 2.2]{GL} for the case $n\geq 3$. It was first proved in the present form by different means in \cite[Section 5]{IYa}.

{\bf  Corollary C.}
{\em \begin{enumerate}[label=(\alph*)]
    \item The category $\mathcal{T}/\mathcal{T}^{fd}$ is Hom-finite and $(n-1)$-Calabi-Yau.
    \item The object $M$ is $(n-1)$-cluster tilting in $\mathcal{T}/\mathcal{T}^{fd}$.
    \item We have that $\mathcal{T}/\mathcal{T}^{fd}(M,M)\cong\mathcal{T}(M,M)$.
\end{enumerate}
}

In Corollary C, parts (a) and (c), corresponding to Corollaries \ref{coro_CY} and \ref{coro_part3} respectively, are consequences of Theorem A, while part (b), corresponding to Corollary \ref{coro_n-1ct}, is a consequence of Theorem B.

The paper is organised as follows. Section \ref{section_defns} introduces some definitions and our setup. Section \ref{section_thmA} presents some results on the Hom-spaces of $\mathcal{T}$ and $\mathcal{T}/\mathcal{T}^{fd}$ and proves Theorem A. Section \ref{section_thmB} proves Theorem B and Section \ref{section_coro} applies the previous sections to prove Corollary C.

\section{Definitions}\label{section_defns}

\begin{defn}
Let $\mathcal{A}$ be an additive category and $\mathcal{B}\subseteq \mathcal{A}$ be a full subcategory. We define the full subcategories
\begin{align*}
    \prescript{\perp_\mathcal{A}}{}{\mathcal{B}}:=\{ A\in\mathcal{A}\mid \Hom_{\mathcal{A}}(A,\mathcal{B})=0 \},\\
    \mathcal{B}^{\perp_{\mathcal{A}}}:=\{ A\in\mathcal{A}\mid \Hom_{\mathcal{A}}(\mathcal{B},A)=0 \}.
\end{align*}
\end{defn}
\begin{defn}
Let $\mathcal{T}$ be a triangulated category and $\mathcal{S}\subseteq \mathcal{T}$ be a full subcategory. We define thick$(\mathcal{S})$ to be the smallest triangulated subcategory of $\mathcal{T}$ closed under direct summands and containing $\mathcal{S}$.
\end{defn}

\begin{defn}[{\cite[Section 2.3]{IYa}}]
Let $\mathcal{T}$ be a triangulated category. A full subcategory $\mathcal{P}$ of $\mathcal{T}$ is called \textit{presilting} if $\Hom_\mathcal{T}(\mathcal{P}, \Sigma^i\mathcal{P})=0$ for any $i>0$. It is called \textit{silting} if in addition $\mathcal{T}=\text{thick}(\mathcal{P})$. An object $P\in\mathcal{T}$ is called \textit{presilting}, respectively \textit{silting}, if add$(P)$ is a presilting, respectively silting, subcategory of $\mathcal{T}$.
\end{defn}

\begin{defn}
A \textit{torsion pair} of a triangulated category $\mathcal{T}$ is a pair $(\mathcal{X}, \mathcal{Y})$ of full subcategories of $\mathcal{T}$ such that $\mathcal{X}= \prescript{\perp_\mathcal{T}}{}{\mathcal{Y}}$, $\mathcal{Y}=\mathcal{X}^{\perp_\mathcal{T}}$ and $\mathcal{T}=\mathcal{X}*\mathcal{Y}$, where
\begin{align*}
  \mathcal{X}*\mathcal{Y}:=\{M\in\mathcal{T}\mid \text{ there is a triangle } X\rightarrow M\rightarrow Y\rightarrow\Sigma X \text{ in } \mathcal{T} \text{ with } X\in\mathcal{X},\,\, Y\in\mathcal{Y}\},
\end{align*}
see \cite[Definition 2.2]{IY}.

A \textit{$t$-structure on $\mathcal{T}$} is a pair $(\mathcal{T}^{\leq 0}, \mathcal{T}^{\geq 0})$ of full subcategories of $\mathcal{T}$ such that $\mathcal{T}^{\geq 1}\subset \mathcal{T}^{\geq 0}$ and $(\mathcal{T}^{\leq 0}, \mathcal{T}^{\geq 1})$ is a torsion pair. Here, for an integer $m$, we denote $\mathcal{T}^{\leq m}=\Sigma^{-m}\mathcal{T}^{\leq 0}$ and $\mathcal{T}^{\geq m}=\Sigma^{-m}\mathcal{T}^{\geq 0}$, see \cite[Definition 1.3.1]{B}.

Similarly, a \textit{co-$t$-structure on $\mathcal{T}$} is a pair $(\mathcal{T}_{\geq 0}, \mathcal{T}_{\leq 0})$ of full subcategories of $\mathcal{T}$ such that $\mathcal{T}_{\geq 1}\subset \mathcal{T}_{\geq 0}$ and $(\mathcal{T}_{\geq 1}, \mathcal{T}_{\leq 0})$ is a torsion pair.  Here, for an integer $m$, we denote $\mathcal{T}_{\geq m}=\Sigma^{-m}\mathcal{T}_{\geq 0}$ and $\mathcal{T}_{\leq m}=\Sigma^{-m}\mathcal{T}_{\leq 0}$, see \cite[Definition 2.4]{P} and \cite[Definition 1.1.1]{Bo}.
\end{defn}

\begin{defn}\label{defn_trunc_triangle}
If $(\mathcal{T}^{\leq 0}, \mathcal{T}^{\geq 0})$ is a $t$-structure on $\mathcal{T}$, then by \cite[Proposition 1.3.3(ii)]{B} for each $X\in\mathcal{T}$, there is a triangle of the form
\begin{align*}
    X^{\leq 0}\xrightarrow{f^0} X\xrightarrow{g^1} X^{\geq 1}\rightarrow\Sigma X^{\leq 0},
\end{align*}
with $X^{\leq 0}\in\mathcal{T}^{\leq 0}$ and $X^{\geq 1}\in\mathcal{T}^{\geq 1}$ which is unique up to unique isomorphism and it is called the \textit{truncation triangle associated to $X$}. Moreover, for an integer $m$, the pair $(\mathcal{T}^{\leq m}, \mathcal{T}^{\geq m})$ is also a $t$-structure with truncation triangle associated with $X$ denoted by
\begin{align*}
    X^{\leq m}\xrightarrow{f^m} X\xrightarrow{g^{m+1}} X^{\geq m+1}\rightarrow\Sigma X^{\leq m}.
\end{align*}
We fix the notation used above for truncation triangles in the rest of the paper.
\end{defn}

\begin{defn}[{\cite[Definition 5.1]{IYa}}]\label{defn_nCYtriple}
Let $n\geq 3$ be an integer. Let $\mathcal{T}$ be a $k$-linear triangulated category, $\mathcal{M}$ be an additive subcategory of $\mathcal{T}$ and $\mathcal{T}^{fd}$ be a triangulated subcategory of $\mathcal{T}$. We say that $(\mathcal{T}, \mathcal{T}^{fd},\mathcal{M})$ is an \textit{$n$-Calabi-Yau triple} if the following are satisfied.
\begin{enumerate}
    \item[(CY1)] The category $\mathcal{T}$ is Hom-finite and Krull-Schmidt.
    \item[(CY2)] The pair $(\mathcal{T},\mathcal{T}^{fd})$ is \textit{relative $n$-Calabi-Yau} in the sense that there exists a bifunctorial isomorphism for any $X\in\mathcal{T}^{fd}$ and $Y\in \mathcal{T}$:
    \begin{align*}
        D\Hom_\mathcal{T}(X,Y)\cong \Hom_\mathcal{T}(Y,\Sigma^n X).
    \end{align*}
    \item[(CY3)] The subcategory $\mathcal{M}$ is a silting subcategory of $\mathcal{T}$ and admits a right adjacent $t$-structure $(\mathcal{T}^{\leq 0},\mathcal{T}^{\geq 0}):=((\Sigma^{<0}\mathcal{M})^{\perp_\mathcal{T}}, (\Sigma^{>0}\mathcal{M})^{\perp_\mathcal{T}})$ with $\mathcal{T}^{\geq 0}\subset\mathcal{T}^{fd}$.
\end{enumerate}
\end{defn}

\begin{remark}\label{remark_AG_setup}
Suppose we have a differential graded (dg) $k$-algebra $A$ with certain properties, see \cite[(1)-(3) and (4') in Section 1]{GL}.  Let $\mathcal{T}=\text{per}(A)$ be the perfect derived category of $A$ and $\mathcal{T}^{fd}=\mathcal{D}^b (A)$ be the full subcategory of the derived category of dg $A$-modules whose objects are the objects with finite-dimensional total homology. Then $(\text{per} (A), \mathcal{D}^b(A), \text{add}(A))$ is an $n$-Calabi-Yau triple.
\end{remark}

In order to illustrate that $n$-Calabi-Yau triples are not uncommon, we consider the following example, showing a class of examples of $3$-Calabi-Yau triples.
\begin{exmp}\label{example_3cy}
Let $(Q,W)$ be a quiver with potential and let $\Gamma=\Gamma(Q,W)$ be its Ginzburg dg algebra, see \cite[Section 4.2]{Gin} for details on this construction. Consider the algebra
\begin{align*}
    J(Q, W):=KQ/\langle  \delta_a W\mid a \text{ is an arrow in } Q \rangle,
\end{align*}
where $\delta_a: KQ/[KQ,KQ]\rightarrow KQ$ is the cyclic derivative with respect to the arrow $a$ and note that $H^0(\Gamma)=J(Q,W)$ by \cite[Definition 3.3]{AC}.  If $H^0(\Gamma)$ is finite-dimensional, then, using the same notation as in Remark \ref{remark_AG_setup}, we have that $(\text{per}(\Gamma), \mathcal{D}^{b}(\Gamma), \text{add}(\Gamma))$ is a $3$-Calabi-Yau triple, see \cite[Remark 5.16]{IYa}.
\end{exmp}

\section{The relation between morphisms in $\mathcal{T}$ and in $\mathcal{T}/\mathcal{T}^{fd}$}\label{section_thmA}

The goal of this section is to prove Theorem \ref{thm_directstab}. Working in Setup \ref{setup}, this shows the relation between Hom-spaces in the triangulated category $\mathcal{T}$ and Hom-spaces in the triangulated quotient category $\mathcal{T}/\mathcal{T}^{fd}$, see \cite[Sections 7-9]{MJ} for details on the construction of this quotient category.
In order to build this relation, we use inverse and direct systems, see \cite[Chapter 5.2]{Rot} for details on these systems.

%From now on, we work in the following setup.

\begin{setup}\label{setup}
Let $k$ be a field, $n\geq 3$ an integer and $(\mathcal{T}, \mathcal{T}^{fd},\mathcal{M})$ be an $n$-Calabi-Yau triple with $\mathcal{M}=\add (M)$ for some object $M\in\mathcal{T}$. Note that, since $\mathcal{M}$ is silting, we have  $\mathcal{T}=\text{thick}\,(M)$.
\end{setup}

\begin{remark}\label{remark_cot_structure}
By \cite[Section 5.1]{IYa}, we have that
\begin{align*}
    (\mathcal{T}_{\geq 0},\mathcal{T}_{\leq 0}):=\Bigg(\bigcup_{i\geq 0} \Sigma^{-i}\mathcal{M}*\Sigma^{-i+1}\mathcal{M}*\cdots*\Sigma^{-1}\mathcal{M}*\mathcal{M}, \bigcup_{i\geq 0} \mathcal{M}*\Sigma\mathcal{M}*\cdots*\Sigma^{i-1}\mathcal{M}*\Sigma^i\mathcal{M} \Bigg)
\end{align*}
is a \textit{bounded co-$t$-structure}, that is
\begin{align*}
   \bigcup_{i\in \mathbb{Z}}\Sigma^i \mathcal{T}_{\geq 0}= \bigcup_{i\in \mathbb{Z}}\Sigma^i \mathcal{T}_{\leq 0}=\mathcal{T}, 
\end{align*}
see \cite[Definition 2.1]{JPcot}. Hence, given an object $X$ in $\mathcal{T}$, there are integers $l$ and $m$ such that $X\in\Sigma^l\mathcal{T}_{\geq 0}$ and $X\in\Sigma^m\mathcal{T}_{\leq 0}$. Moreover, by \cite[p. 7886, line 2]{IYa}, we have that $\mathcal{T}^{\leq 0}=\mathcal{T}_{\leq 0}$ and similarly $\mathcal{T}^{\leq i}=\mathcal{T}_{\leq i}$ for any integer $i$.
\end{remark}

\begin{lemma}\label{lemma_inversesys}
For any $X\in\mathcal{T}$, there is a canonical inverse system of the form
\begin{align*}
\xymatrix@C=3em{
    \cdots\ar[r]^{\xi^{-4}}&X^{\leq -3}\ar[r]^{\xi^{-3}}& X^{\leq -2}\ar[r]^{\xi^{-2}} &X^{\leq -1}\ar[r]^{\xi^{-1}} &X^{\leq 0}\ar[r]^{f^{0}}& X.
    }
\end{align*}
\end{lemma}

\begin{proof}
Note that we have a chain of inclusions of full subcategories
\begin{align*}
    \cdots\subset \mathcal{T}^{\leq -3}\subset \mathcal{T}^{\leq -2}\subset \mathcal{T}^{\leq -1}\subset \mathcal{T}^{\leq 0}.
\end{align*}
Given $X\in \mathcal{T}$ and $p\geq 1$, consider the truncation triangles associated to $X$ with respect to the $t$-structures $(\mathcal{T}^{\leq -p}, \mathcal{T}^{\geq -p})$ and $(\mathcal{T}^{\leq -p+1}, \mathcal{T}^{\geq -p+1})$. Then there is a morphism of triangles of the form
\begin{align}\label{diagram_xi}
    \xymatrix@C=3em{
    X^{\leq -p}\ar[r]^{f^{-p}}\ar[d]^{\xi^{-p}}& X\ar[r]\ar@{=}[d]& X^{\geq -p+1}\ar[r]\ar@{-->}[d]&\Sigma X^{\leq -p}\ar[d]\\
    X^{\leq -p+1}\ar[r]_{f^{-p+1}}& X\ar[r]& X^{\geq -p+2}\ar[r]&\Sigma X^{\leq -p+1},
    }
\end{align}
where $\xi^{-p}$ such that $f^{-p+1}\circ \xi^{-p}=f^{-p}$ exists since $X^{\leq -p}\in\mathcal{T}^{\leq -p}\subset \mathcal{T}^{\leq -p+1}$, $X^{\geq -p+2}\in\mathcal{T}^{\geq -p+2}$ and $(\mathcal{T}^{\leq -p+1},\mathcal{T}^{\geq -p+2})$ is a torsion pair and the dashed morphism then exists by the axioms of triangulated categories. Then
\begin{align*}
\xymatrix@C=3em{
    \cdots\ar[r]^{\xi^{-4}}&X^{\leq -3}\ar[r]^{\xi^{-3}}& X^{\leq -2}\ar[r]^{\xi^{-2}} &X^{\leq -1}\ar[r]^{\xi^{-1}} &X^{\leq 0}\ar[r]^{f^{0}}& X
    }
\end{align*}
is an inverse system.
\end{proof}

\begin{lemma}\label{lemma_iso_quotient_cat}
Let $X\in\mathcal{T}$. For any integer $m$, in the quotient category $\mathcal{T}/\mathcal{T}^{fd}$ we have that $f^m$ and $\xi^m$ become  isomorphisms and $X\cong_{\mathcal{T}/\mathcal{T}^{fd}} X^{\leq m}$.
\end{lemma}
\begin{proof}
By (CY3), we have that $\mathcal{T}^{\geq 0}\subset \mathcal{T}^{fd}$. Since $\mathcal{T}^{fd}$ is triangulated and hence closed under integer powers of $\Sigma$, we have that $\mathcal{T}^{\geq m+1}\subset \mathcal{T}^{fd}$ for any integer $m$. Then the truncation triangle
\begin{align*}
    X^{\leq m}\xrightarrow{f^m} X\xrightarrow{g^{m+1}} X^{\geq m+1}\rightarrow\Sigma X^{\leq m}
\end{align*}
viewed as a triangle in $\mathcal{T}/\mathcal{T}^{fd}$ is such that $X^{\geq m+1}\cong_{\mathcal{T}/\mathcal{T}^{fd}} 0$ and $f^m$ is an isomorphism. Hence $X\cong_{\mathcal{T}/\mathcal{T}^{fd}} X^{\leq m}$. Since $f^{m+1}\circ \xi^m=f^m$, while $f^{m+1}$ and $f^{m}$ become isomorphisms in $\mathcal{T}/\mathcal{T}^{fd}$, so does $\xi^m$.
\end{proof}

\begin{remark}\label{remark_psi}
Given $X$ and $Y$ in $\mathcal{T}$, applying the functor $\mathcal{T}(-,Y)$ to the inverse system from Lemma \ref{lemma_inversesys}, we obtain a direct system of the form
\begin{align*}
\xymatrix@C=3em{
    \mathcal{T}(X,Y)\ar[r]&\mathcal{T}(X^{\leq 0},Y)\ar[r]&\mathcal{T}(X^{\leq -1},Y)\ar[r]&\mathcal{T}(X^{\leq -2},Y)\ar[r]&\cdots.
    }
\end{align*}
Moreover, passing to the quotient category $\mathcal{T}/\mathcal{T}^{fd}$ using the quotient functor $Q:\mathcal{T}\rightarrow \mathcal{T}/\mathcal{T}^{fd}$, we obtain another direct system and a commutative diagram of the form
\begin{align*}
\xymatrix@C=2.5em{
    \mathcal{T}(X,Y)\ar[r]\ar[d]^{Q(-)}&\mathcal{T}(X^{\leq 0},Y)\ar[r]\ar[d]^{Q(-)}&\mathcal{T}(X^{\leq -1},Y)\ar[r]\ar[d]^{Q(-)}&\mathcal{T}(X^{\leq -2},Y)\ar[r]\ar[d]^{Q(-)}&\cdots\\
    \mathcal{T}/\mathcal{T}^{fd}(X,Y)\ar[r]^{\sim}&\mathcal{T}/\mathcal{T}^{fd}(X^{\leq 0},Y)\ar[r]^{\sim}&\mathcal{T}/\mathcal{T}^{fd}(X^{\leq -1},Y)\ar[r]^{\sim}&\mathcal{T}/\mathcal{T}^{fd}(X^{\leq -2},Y)\ar[r]^-{\sim}&\cdots,
    }
\end{align*}
where all the arrows in the bottom row are isomorphisms by Lemma \ref{lemma_iso_quotient_cat}.
Then, by the universal property of direct systems, there exists a unique morphism of the form
\begin{align*}
    \Psi: \underset{q}{\varinjlim}\,\mathcal{T}(X^{\leq -q}, Y)\rightarrow \underset{q}{\varinjlim}\,\mathcal{T}/\mathcal{T}^{fd}(X^{\leq -q}, Y)
\end{align*}
such that the diagram
\begin{align*}
    \xymatrix{
    \mathcal{T}(X^{\leq -p}, Y)\ar[d]\ar[r]^{Q(-)}&\mathcal{T}/\mathcal{T}^{fd}(X^{\leq -p}, Y)\ar[d]^-{\rotatebox{90}{$\sim$}}\\
    \underset{q}{\varinjlim}\,\mathcal{T}(X^{\leq -q}, Y)\ar@{-->}[r]^-{\Psi}&\underset{q}{\varinjlim}\,\mathcal{T}/\mathcal{T}^{fd}(X^{\leq -q}, Y)
    }
\end{align*}
commutes for every $p\geq 0$.
\end{remark}

\begin{lemma}\label{lemma_psi_iso}
The morphism
\begin{align*}
    \Psi: \underset{q}{\varinjlim}\,\mathcal{T}(X^{\leq -q}, Y)\rightarrow \underset{q}{\varinjlim}\,\mathcal{T}/\mathcal{T}^{fd}(X^{\leq -q}, Y)
\end{align*}
from Remark \ref{remark_psi} is an isomorphism.
\end{lemma}
\begin{proof}
We first prove that $\Psi$ is surjective. Let $\gamma$ be an element in $\varinjlim\,\mathcal{T}/\mathcal{T}^{fd}(X^{\leq -q}, Y)$. Note that by \cite[Lemma 5.30(i)]{Rot}, we have that $\gamma$ comes from an element in one of the Hom-spaces of the direct system and since these are all isomorphic, we may assume $\gamma$ comes from an element of the form
\begin{align*}
    \alpha=\left[ \vcenter{\xymatrix@R=1em @C=1em{
& Z&
\\ X\ar[ru]^h& &Y\ar[lu]_{s}}}\right]\in\mathcal{T}/\mathcal{T}^{fd}(X,Y),
\end{align*}
where the triangle extending $s$, say
\begin{align*}
    \xymatrix@C=2.5em{
    W\ar[r]&Y\ar[r]^-s&Z\ar[r]^-{g}&\Sigma W,
    }
\end{align*}
has $W$ and $\Sigma W$ in $\mathcal{T}^{fd}$ since $s$ is in the multiplicative system being inverted. By \cite[Lemma 4.11]{IYa} there exists an integer $p$ such that
\begin{align*}
    \Sigma W\in(\Sigma^{>p}\mathcal{M})^{\perp_\mathcal{T}}=\mathcal{T}^{\geq -p}=(\mathcal{T}^{\leq -p-1})^{\perp_\mathcal{T}}.
\end{align*}
Hence $X^{\leq -p-1}\in\mathcal{T}^{\leq -p-1}$ is such that $\mathcal{T}(X^{\leq -p-1}, \Sigma W)=0$.
Then $g\circ h\circ f^{-p-1}=0$, and we have a commutative diagram
\begin{align*}
    \xymatrix@C=5em{
    && X^{\leq -p-1}\ar[d]^-{h\circ f^{-p-1}}\ar[rd]^0\ar@{-->}[ld]_{\exists\, l}\\
    W\ar[r]&Y\ar[r]_s&Z\ar[r]_{g}&\Sigma W.
    }
\end{align*}
Since $s\circ l= h\circ f^{-p-1}$, we have that
\begin{align*}
    Q(l)= \left[ \vcenter{\xymatrix@R=1em @C=1em{
& Y&
\\ X^{\leq -p-1}\ar[ru]^l& &Y\ar@{=}[lu]}}\right]=
\left[ \vcenter{\xymatrix@R=1em @C=1em{
& Z&
\\ X^{\leq -p-1}\ar[ru]^{h\circ f^{-p-1}}& &Y\ar[lu]_{s}}}\right].
\end{align*}
Consider the commutative diagram
\begin{align}\label{diagram_lim_comm}
    \xymatrix{
    &\mathcal{T}/\mathcal{T}^{fd}(X, Y)\ar[d]_-{\rotatebox{90}{$\sim$}}^{\overline{\mu}}\\
    \mathcal{T}(X^{\leq -p}, Y)\ar[d]_{\nu^{-p}}\ar[r]^-{Q(-)}&\mathcal{T}/\mathcal{T}^{fd}(X^{\leq -p}, Y)\ar[d]_-{\rotatebox{90}{$\sim$}}^{{\overline{\nu}^{-p}}}\\
    \underset{q}{\varinjlim}\,\mathcal{T}(X^{\leq -q}, Y)\ar[r]^-{\Psi}&\underset{q}{\varinjlim}\,\mathcal{T}/\mathcal{T}^{fd}(X^{\leq -q}, Y).
    }
\end{align}
Then, we have
\begin{align*}
    \gamma=\overline{\nu}^{\,-p}\circ\overline{\mu}(\alpha)=\overline{\nu}^{-p}\circ Q(l)=\Psi\circ \nu^{-p}(l),
\end{align*}
so $\Psi$ is surjective.

It remains to show that $\Psi$ is also injective. Consider an element $l$ in $\varinjlim\,\mathcal{T}(X^{\leq -q}, Y)$ such that $\Psi(l)=0$. Note that $l$ comes from an element in one of the Hom-spaces of the direct system, say from the element $h\in\mathcal{T}(X^{\leq -p}, Y)$. Then, considering the commutative diagram (\ref{diagram_lim_comm}), we have 
\begin{align*}
    \overline{\nu}^{-p}\circ Q(h)=\Psi\circ \nu^{-p}(h)=\Psi(l)=0.
\end{align*}
Since $\overline{\nu}^{-p}$ is an isomorphism, we have that $Q(h)=0$. Hence there exists a morphism $t: Y\rightarrow K$ in the multiplicative system being inverted when we pass to $\mathcal{T}/\mathcal{T}^{fd}$ such that $t\circ h=0$. Consider the triangle extending $t$, say
\begin{align*}
    \xymatrix@C=2.5em{
    W\ar[r]^w &Y\ar[r]^t&K\ar[r]&\Sigma W,
    }
\end{align*}
where $W\in\mathcal{T}^{fd}$. Since $t\circ h=0$, there exists a morphism $g:X^{\leq - p}\rightarrow W$ such that $h=w\circ g$, that is such that the following commutes:
\begin{align*}
    \xymatrix@C=5em{
    & X^{\leq -p}\ar[d]^h\ar[rd]^0\ar@{-->}[ld]_g\\
    W\ar[r]^w &Y\ar[r]^t&K\ar[r]&\Sigma W.
    }
\end{align*}
Moreover, as $W\in\mathcal{T}^{fd}$, by \cite[Lemma 4.11]{IYa} there is an integer $i$ such that
\begin{align*}
    W\in(\Sigma^{>i}\mathcal{M})^{\perp_\mathcal{T}}=\mathcal{T}^{\geq -i}=(\mathcal{T}^{\leq -i-1})^{\perp_\mathcal{T}}.
\end{align*}
Since $X^{\leq -i-1}\in\mathcal{T}^{\leq -i-1}$, we have that $\mathcal{T}(X^{\leq -i-1}, W)=0$. We now consider two cases. First, if $p-1\geq i$, we have that $W\in \mathcal{T}^{\geq -p+1}=(\mathcal{T}^{\leq -p})^{\perp_\mathcal{T}}$, so that $\mathcal{T}(X^{\leq -p}, W)=0$ and $h=w\circ g=0$ implying that $l=\nu^{-p}(h)=0$. In the other case, that is $p-1< i$, the inverse system gives us a morphism $\xi:X^{\leq -i-1}\rightarrow X^{\leq -p}$. Then, as $\mathcal{T}(X^{\leq -i-1}, W)=0$, we have that $g\circ \xi=0$ and so $h\circ\xi=w\circ g\circ \xi=0$. Consider the commutative diagram
\begin{align*}
    \xymatrix@C=3em{
    \mathcal{T}(X^{\leq -p}, Y)\ar[r]^{\xi^*}\ar[rd]_{\nu^{-p}}& \mathcal{T}(X^{\leq -i-1}, Y)\ar[d]^{\nu^{-i-1}}\\
    & \underset{q}{\varinjlim}\,\mathcal{T}(X^{\leq -q}, Y).
    }
\end{align*}
We have that $0=\nu^{-i-1}(0)=\nu^{-i-1}(h\circ\xi)=\nu^{-i-1}\circ\xi^*(h)=\nu^{-p}(h)=l$. Hence $\Psi$ is injective.
\end{proof}

\begin{notation}\label{notation_Cpx}
Given $X\in\mathcal{T}$, note that diagram (\ref{diagram_xi}) from the proof of Lemma \ref{lemma_inversesys} can also be built for non-positive integers. Then, for any integer $p$, the morphism $\xi^{-p}$ is defined. Let the triangle in $\mathcal{T}$ extending $\xi^{-p}$ be
\begin{align*}
    \xymatrix@C=3em{
    X^{\leq -p}\ar[r]^{\xi^{-p}}& X^{\leq -p+1}\ar[r]&C^{-p}_X\ar[r]& \Sigma X^{\leq -p}.
    }
\end{align*}
\end{notation}

\begin{lemma}\label{lemma_Cp}
Given $X\in\mathcal{T}$ and an integer $p$, we have that the object $C^{-p}_X$ is in $\mathcal{T}^{-p+1}:=\mathcal{T}^{\leq -p+1}\cap\mathcal{T}^{\geq -p+1}$. In particular, $C^{-p}_X\in\mathcal{T}^{fd}$ by (CY3).
\end{lemma}
\begin{proof}
Consider the truncation triangles associated to $X$ with respect to the $t$-structures $(\mathcal{T}^{\leq -p},\mathcal{T}^{\geq -p})$ and $(\mathcal{T}^{\leq -p+1},\mathcal{T}^{\geq -p+1})$ from Definition \ref{defn_trunc_triangle}. By the octahedral axiom, we have a commutative diagram of triangles in $\mathcal{T}$ of the form
\begin{align}\label{diagram_Cpx}
   \xymatrix@C=3em{
    X^{\leq -p}\ar[r]^{f^{-p}}\ar[d]^{\xi^{-p}}& X\ar[r]\ar@{=}[d]& X^{\geq -p+1}\ar[r]\ar[d]&\Sigma X^{\leq -p}\ar[d]\\
    X^{\leq -p+1}\ar[r]^{f^{-p+1}}\ar[d]& X\ar[r]\ar[d]& X^{\geq -p+2}\ar[r]\ar[d]&\Sigma X^{\leq -p+1}\ar[d]\\
    C^{-p}_X\ar[r]\ar[d]&0\ar[r]\ar[d]& \Sigma C^{-p}_X\ar@{=}[r]\ar[d]&\Sigma C^{-p}_X\ar[d]\\
    \Sigma X^{\leq -p}\ar[r]&\Sigma X\ar[r]& \Sigma X^{\geq -p+1}\ar[r]&\Sigma^2 X^{\leq -p}.
    }
\end{align}
Note that
\begin{align*}
    X^{\leq -p}\in \mathcal{T}^{\leq -p}, \, X^{\leq -p+1}\in \mathcal{T}^{\leq -p+1} \text{ and } \Sigma X^{\leq -p}\in \Sigma\mathcal{T}^{\leq -p}=\mathcal{T}^{\leq -p-1}.
\end{align*}
Since $\mathcal{T}^{\leq -p-1}\subseteq \mathcal{T}^{\leq -p}\subseteq \mathcal{T}^{\leq -p+1}$, and $\mathcal{T}^{\leq -p+1}$ is closed under extensions, we have that $C^{-p}_X\in \mathcal{T}^{\leq -p+1}$.
Moreover, note that
\begin{align*}
    X^{\geq -p+1}\in \mathcal{T}^{\geq -p+1}, \, X^{\geq -p+2}\in \mathcal{T}^{\geq -p+2} \text{ and } \Sigma X^{\geq -p+1}\in \Sigma\mathcal{T}^{\geq -p+1}=\mathcal{T}^{\geq -p}.
\end{align*}
Since $\mathcal{T}^{\geq -p+2}\subseteq \mathcal{T}^{\geq -p+1}\subseteq \mathcal{T}^{\geq -p}$, and $\mathcal{T}^{\geq -p}$ is closed under extensions, we have that $\Sigma C^{-p}_X\in \mathcal{T}^{\geq -p}=\Sigma \mathcal{T}^{\geq -p+1}$ and so $C^{-p}_X\in \mathcal{T}^{\geq -p+1}$.
\end{proof}

\begin{lemma}\label{lemma_pbig_Y_thick}
Let $X$ and $Y$ be objects in $\mathcal{T}$.
\begin{enumerate}[label=(\alph*)]
\item  Let $p\gg 0$ be an integer, then $\mathcal{T}(Y,C^{-p}_X)=0$.
\item Let $q\gg0$ be an integer, then $\mathcal{T}(X, Y^{\leq -q})=0$.
\end{enumerate}
\end{lemma}
\begin{proof}
(a) By Remark \ref{remark_cot_structure}, $(\mathcal{T}_{\geq 0}, \mathcal{T}_{\leq 0})$ is a bounded co-$t$-structure. Hence there is an integer $i$ such that $Y\in\Sigma^{-i}\mathcal{T}_{\geq 0}=\mathcal{T}_{\geq i}$. Pick an integer $p\geq -i+2$. Then $-p+1\leq i-1$ and, using Lemma \ref{lemma_Cp}, we have that
\begin{align*}
C^{-p}_X\in\mathcal{T}^{\leq -p+1}\subseteq \mathcal{T}^{\leq i-1}.
\end{align*}
Moreover, by Remark \ref{remark_cot_structure}, we have that $\mathcal{T}^{\leq i-1}=\mathcal{T}_{\leq i-1}$, and so $C^{-p}_X\in\mathcal{T}_{\leq i-1}$. Hence, as $(\mathcal{T}_{\geq i}, \mathcal{T}_{\leq i-1})$ is a torsion pair, we have that $\mathcal{T}(Y,C^{-p}_X)=0$.

(b) By Remark \ref{remark_cot_structure}, there is an integer $i$ such that $X\in \Sigma^{-i}\mathcal{T}_{\geq 0}=\mathcal{T}_{\geq i}$. Pick an integer $q>-i$. Then $-q\leq i-1$ and we have that
\begin{align*}
    Y^{\leq -q}\in \mathcal{T}^{\leq -q}\subseteq \mathcal{T}^{\leq i-1}=\mathcal{T}_{\leq i-1},
\end{align*}
where the last equality holds by Remark \ref{remark_cot_structure}. Hence, as $(\mathcal{T}_{\geq i}, \mathcal{T}_{\leq i-1})$ is a torsion pair, we have that $\mathcal{T}(X,Y^{\leq -q})=0$.
\end{proof}

\begin{theorem}\label{thm_directstab}
Let $X$ and $Y$ be objects in $\mathcal{T}$.
\begin{enumerate}[label=(\alph*)]
\item For $p\gg0$, the direct system
\begin{align*}
\xymatrix@C=3em{
    \mathcal{T}(X,Y)\ar[r]&\mathcal{T}(X^{\leq 0},Y)\ar[r]&\mathcal{T}(X^{\leq -1},Y)\ar[r]&\mathcal{T}(X^{\leq -2},Y)\ar[r]&\cdots
    }
\end{align*}
stabilizes. Moreover, we have that
    $\underset{q}{\varinjlim}\,\mathcal{T}(X^{\leq -q}, Y)\cong \mathcal{T}/\mathcal{T}^{fd}(X, Y)$.
\item For $X$ and $Y$ in $\mathcal{T}$, we have that
    \begin{align*}
       \underset{q}{\varprojlim} \Big(\underset{p}{\varinjlim}\,\mathcal{T}(\Sigma^{-1}X^{\geq -p+1},Y^{\leq -q})\Big)\cong \mathcal{T}/\mathcal{T}^{fd}(X,Y).
    \end{align*}
\end{enumerate}
\end{theorem}

\begin{remark}
Note that the above result resembles some previous results. Despite the different setup, it is worth pointing out the similarity between the formula in part (a) and (2.2.1) from \cite[proof of Proposition 2.2]{AZ}. Moreover, the fact that the direct system in (a) stabilizes resembles \cite[Proposition 3.13]{AZ}. 
\end{remark}

\begin{proof}[Proof of Theorem \ref{thm_directstab}]
(a) Consider the triangle
\begin{align*}
\xymatrix@C=3em{
    X^{\leq -p}\ar[r]^{\xi^{-p}}& X^{\leq -p+1}\ar[r]&C^{-p}_X\ar[r]& \Sigma X^{\leq -p}
    }
\end{align*}
and the exact sequence obtained by applying the functor $\mathcal{T}(-,Y)$ to it:
\begin{align*}
\xymatrix@C=3em{
    \mathcal{T}(C^{-p}_X,Y)\ar[r]&\mathcal{T}(X^{\leq -p+1},Y)\ar[r]&\mathcal{T}(X^{\leq -p},Y)\ar[r]&\mathcal{T}(\Sigma^{-1}C^{-p}_X,Y).
    }
\end{align*}
Since $C^{-p}_X\in\mathcal{T}^{fd}$, we have $\mathcal{T}(C^{-p}_X,Y)\cong D\mathcal{T}(Y,\Sigma^{n}C^{-p}_X)$ and $\mathcal{T}(\Sigma^{-1}C^{-p}_X,Y)\cong D\mathcal{T}(Y,\Sigma^{n-1}C^{-p}_X)$. For $p\gg0$, by Lemma \ref{lemma_pbig_Y_thick}(a), we have that $\mathcal{T}(X^{\leq -p+1},Y)\cong \mathcal{T}(X^{\leq -p},Y)$. Hence the direct system stabilizes as claimed and for $p\gg0$ we have that
\begin{align*}
\underset{q}{\varinjlim}\,\mathcal{T}(X^{\leq -q}, Y)\cong \mathcal{T}(X^{\leq -p}, Y).
\end{align*}
Then, since $\Psi$ is an isomorphism by Lemma \ref{lemma_psi_iso}, for $p\gg0$ we obtain the isomorphism
\begin{align*}
    \xymatrix@C=5em{
    \underset{q}{\varinjlim}\,\mathcal{T}(X^{\leq -q}, Y)\ar[r]^-{\Psi}_-{\sim}& \underset{q}{\varinjlim}\,\mathcal{T}/\mathcal{T}^{fd}(X^{\leq -q}, Y)\ar[r]_-{\sim}^-{(\overline{\mu})^{-1}\circ ({\overline{\nu}^{-p}})^{-1}}& 
    \mathcal{T}/\mathcal{T}^{fd}(X, Y)}.
\end{align*}

(b) Let $X$ and $Y$ be objects in $\mathcal{T}$ and $p,\, q$ be integers. Applying the functor $\mathcal{T}(-,Y^{\leq -q})$ to the truncation triangle
\begin{align}\label{diagram_triangleX}
    \xymatrix@C=3em{
    X^{\leq -p}\ar[r]& X\ar[r]&X^{\geq -p+1}\ar[r]& \Sigma X^{\leq -p},
    }
\end{align}
we obtain the exact sequence
\begin{align}\label{diagram_prop1}
    \xymatrix@C=2.5em{
    \mathcal{T}(X,Y^{\leq -q})\ar[r]&\mathcal{T}(X^{\leq -p},Y^{\leq -q})\ar[r]&\mathcal{T}(\Sigma^{-1}X^{\geq -p+1},Y^{\leq -q})\ar[r]&\mathcal{T}(\Sigma^{-1}X,Y^{\leq -q}).
    }
\end{align}
Taking the direct limit of (\ref{diagram_prop1}) with respect to $p$, we obtain the exact sequence
\begin{align}\label{diagram_prop4}
     \xymatrix@C=2.5em{
    \mathcal{T}(X,Y^{\leq -q})\ar[r]&\mathcal{T}/\mathcal{T}^{fd}(X,Y)\ar[r]&\underset{p}{\varinjlim}\,\mathcal{T}(\Sigma^{-1}X^{\geq -p+1},Y^{\leq -q})\ar[r]&\mathcal{T}(\Sigma^{-1}X,Y^{\leq -q}).
    }
\end{align}
In the above, the first and last terms are unchanged since they do not depend on $p$. Moreover, for the second term, we observed that
\begin{align*}
    \underset{p}{\varinjlim}\, \mathcal{T}(X^{\leq -p}, Y^{\leq -q})\cong \mathcal{T}/\mathcal{T}^{fd}(X,Y^{\leq -q})\cong \mathcal{T}/\mathcal{T}^{fd}(X,Y),
\end{align*}
where we used part (a) for the first isomorphism and Lemma \ref{lemma_iso_quotient_cat} for the second one.
For $q\gg0$, by Lemma \ref{lemma_pbig_Y_thick}(b), we have that
\begin{align*}
    \mathcal{T}(X,Y^{\leq -q})=0 \text{ and } \mathcal{T}(\Sigma^{-1}X,Y^{\leq -q})=0.
\end{align*}
Hence, for $q\gg0$, we have that (\ref{diagram_prop4}) is the exact sequence
\begin{align*}
     \xymatrix@C=2.5em{
    0\ar[r]&\mathcal{T}/\mathcal{T}^{fd}(X,Y)\ar[r]^-{\sim}&\underset{p}{\varinjlim}\,\mathcal{T}(\Sigma^{-1}X^{\geq -p+1},Y^{\leq -q})\ar[r]&0.
    }
\end{align*}
By \cite[Exercise 5.22(ii)]{Rot}, chopping off the tail of an inverse system, we obtain the same inverse limit. Hence, we have that
\begin{align*}
       \underset{q}{\varprojlim} \Big(\underset{p}{\varinjlim}\,\mathcal{T}(\Sigma^{-1}X^{\geq -p+1},Y^{\leq -q})\Big)\cong \mathcal{T}/\mathcal{T}^{fd}(X,Y).
    \end{align*}
\end{proof}
When $Y$ is a shift of $M$, we can say more about when the direct system in Theorem \ref{thm_directstab}(a) stabilizes and we obtain the following result.

\begin{corollary}\label{coro_valueforp}
Let $X\in\mathcal{T}$. % be such that $X^{\geq p}=0$ for $p\geq 2$ (NOT NEEDED I THINK)
Then, for any integer $j$, we have that
\begin{align*}
    \mathcal{T}/\mathcal{T}^{fd}(X,\Sigma^j M)\cong \mathcal{T}(X^{\leq -j+n-2}, \Sigma^j M).
\end{align*}
\end{corollary}

\begin{proof}
By Theorem \ref{thm_directstab}(a), for $p\gg0$, we have that $\mathcal{T}/\mathcal{T}^{fd}(X,\Sigma^j M)\cong \mathcal{T}(X^{\leq -p}, \Sigma^j M)$. We show that $p=j-n+2$ is big enough. For any integer $l$, consider the triangle 
\begin{align*}
\xymatrix@C=3em{
    X^{\leq -l}\ar[r]& X^{\leq -l+1}\ar[r]& C^{-l}_X\ar[r]& \Sigma X^{\leq -l},
    }
\end{align*}
where $C^{-l}_X\in\mathcal{T}^{-l+1}$. Applying $\mathcal{T}(-, \Sigma^j M)$ to this triangle, we obtain the exact sequence
\begin{align*}
    \xymatrix{
    \mathcal{T}(C^{-l}_X, \Sigma^j M)\ar[r]&\mathcal{T}(X^{\leq -l+1},\Sigma^j M)\ar[r]^\alpha& \mathcal{T}(X^{\leq -l}, \Sigma^j M)\ar[r]& \mathcal{T}(\Sigma^{-1} C^{-l}_X, \Sigma^j M).
    }
\end{align*}
Since $C^{-l}_X\in\mathcal{T}^{fd}$, we have that
\begin{align*}
\mathcal{T}(C^{-l}_X, \Sigma^j M)\cong D\mathcal{T}(\Sigma^{j-n}M, C^{-l}_X) \text{ and }\mathcal{T}(\Sigma^{-1} C^{-l}_X, \Sigma^j M)\cong D\mathcal{T}( \Sigma^{j-n+1} M, C^{-l}_X).
\end{align*}
Moreover, since $C^{-l}_X\in\mathcal{T}^{-l+1}$, we have that $\mathcal{T}(\Sigma^m M,C^{-l}_X)=0$ for $m<l-1$. Hence, if $l>j-n+2$, the morphism $\alpha$ is an isomorphism.
\end{proof}

Moreover, if we also fix $X=M$, then Corollary \ref{coro_valueforp} has the following important special case. 
\begin{corollary}\label{coro_MandSigma_zero}
We have that $\mathcal{T}/\mathcal{T}^{fd}(M,\Sigma^j M)=0$ for $j=1,\dots, n-2$.
\end{corollary}

In order to prove the above result, we first prove the following.
\begin{lemma}\label{lemma_M_iso_M_leq}
Let $i\geq 0$ be an integer. Then $M\cong_{\mathcal{T}}M^{\leq i}$.
\end{lemma}
\begin{proof}
Consider the truncation triangle associated to $M$ with respect to the $t$-structure $(\mathcal{T}^{\leq i}, \mathcal{T}^{\geq i})$, that is
\begin{align*}
    \xymatrix@C=3em{
    M^{\leq i}\ar[r]& M\ar[r]&M^{\geq i+1}\ar[r]& \Sigma M^{\leq i},
    }
\end{align*}
and recall that it is unique up to unique isomorphism by Definition \ref{defn_trunc_triangle}. Since $M$ is a silting object and $i\geq 0$, we have that
\begin{align*}
    M\in (\Sigma^{<0}\mathcal{M})^{\perp_{\mathcal{T}}}=\mathcal{T}^{\leq 0}\subseteq \mathcal{T}^{\leq i}.
\end{align*}
Hence, the triangle
\begin{align*}
   \xymatrix@C=3em{
    M\ar[r]^{1_M}& M\ar[r]&0\ar[r]& \Sigma M
    }
\end{align*}
is also such that $M\in\mathcal{T}^{\leq i}$ and $0\in\mathcal{T}^{\geq i+1}$. By uniqueness of truncation triangles, we then have that $M\cong M^{\leq i}$. 
\end{proof}

\begin{proof}[Proof of Corollary \ref{coro_MandSigma_zero}]
We have that
\begin{align*}
    \mathcal{T}/\mathcal{T}^{fd}(M,\Sigma^j M)\cong \mathcal{T}(M^{\leq -j+n-2}, \Sigma^j M)\cong \mathcal{T}(M,\Sigma^j M)=0,
\end{align*}
where the first isomorphism holds by Corollary \ref{coro_valueforp} with $X=M$, the second by Lemma \ref{lemma_M_iso_M_leq} and since $-j+n-2\geq -n+2+n-2=0$, and the last equality holds because $M$ is a silting object and $j>0$.
\end{proof}

\section{The Gap Theorem}\label{section_thmB}
In this section, we introduce the ``Gap Theorem". This was inspired by \cite[(2.5) Gap Theorem]{FJ} which, however, relies on a different setup. In the theorem, the vanishing condition for $\mathcal{T}(M,\Sigma^j X)$ can be viewed as a gap in the cohomology of $X$.
\begin{theorem}\label{thm_gapinX}
Let $a$ be an integer and $X\in\mathcal{T}$ be such that $\mathcal{T}(M,\Sigma^j X)=0$ for $a\leq j\leq a+n-2$. Then, the truncation triangle
\begin{align}\label{trunc_tr_a}
    \xymatrix@C=3em{
    X^{\leq a-1}\ar[r]& X\ar[r]&X^{\geq a}\ar[r]^-\delta& \Sigma X^{\leq a-1}
    }
\end{align}
splits and $X\cong X^{\leq a-1}\oplus X^{\geq a}$.
\end{theorem}

\begin{proof}
We first show that $X^{\geq a}\in\mathcal{T}^{\geq a+n-1}$. Since we already know that $X^{\geq a}\in\mathcal{T}^{\geq a}$, it is enough to prove that  $\mathcal{T}(\Sigma^{-j}M,X^{\geq a})=0$ for $a\leq j\leq a+n-2$. Consider the exact sequence obtained by applying $\mathcal{T}(\Sigma^{-j} M,-)$ to (\ref{trunc_tr_a}):
\begin{align*}
     \xymatrix{
    \mathcal{T}(\Sigma^{-j} M, X)\ar[r]&\mathcal{T}(\Sigma^{-j} M, X^{\geq a} )\ar[r]& \mathcal{T}(\Sigma^{-j} M,\Sigma X^{\leq a-1}).
    }
\end{align*}
Note that, for $a\leq j\leq a+n-2$, we have
\begin{align*}
    \mathcal{T}(\Sigma^{-j} M, X)\cong \mathcal{T}( M,\Sigma^j X)=0
\end{align*}
by assumption. Moreover,
\begin{align*}
    \mathcal{T}(\Sigma^{-j} M,\Sigma X^{\leq a-1})\cong \mathcal{T}(\Sigma^{-j-1} M, X^{\leq a-1})=0,
\end{align*}
since $-a-n+1\leq -j-1\leq -a-1$ and $X^{\leq a-1}\in\mathcal{T}^{\leq a-1}$.
Hence $\mathcal{T}(\Sigma^{-j} M, X^{\geq a} )=0$ for $a\leq j\leq a+n-2$ whence $X^{\geq a}\in\mathcal{T}^{\geq a+n-1}$.
Since $X^{\geq a}\in\mathcal{T}^{fd}$, we have that
\begin{align*}
    \delta\in\mathcal{T}(X^{\geq a}, \Sigma X^{\leq a-1})\cong D\mathcal{T}(\Sigma X^{\leq a-1}, \Sigma^n X^{\geq a})=0,
\end{align*}
since $\Sigma X^{\leq a-1}\in \Sigma\mathcal{T}^{\leq a-1}=\mathcal{T}^{\leq a-2}$ and $\Sigma^n X^{\geq a}\in\Sigma^n \mathcal{T}^{\geq a+n-1}=\mathcal{T}^{\geq a-1}=(\mathcal{T}^{\leq a-2})^{\perp_\mathcal{T}}$.
Hence (\ref{trunc_tr_a}) splits and $X\cong X^{\leq a-1}\oplus X^{\geq a}$.
\end{proof}

\begin{remark}
Since we are working in the case $\mathcal{M}=\text{add} (M)$, by \cite[Section 5.1]{IYa}, we have that condition (CY3) from Definition \ref{defn_nCYtriple} is equivalent to its dual:

(CY3$^{\text{op}}$) The subcategory $\mathcal{M}$ is a silting subcategory of $\mathcal{T}$ and admits a left adjacent $t$-structure $(\widetilde{\mathcal{T}}^{\leq 0},\widetilde{\mathcal{T}}^{\geq 0}):=({}^{\perp_{\mathcal{T}}}(\Sigma^{<0}\mathcal{M}),{}^{\perp_{\mathcal{T}}}(\Sigma^{>0}\mathcal{M}))$ with $\widetilde{\mathcal{T}}^{\leq 0}\subseteq \mathcal{T}^{fd}$.

Moreover, note that for an integer $m$, the pair $(\widetilde{\mathcal{T}}^{\leq m},\widetilde{\mathcal{T}}^{\geq m}):=({}^{\perp_{\mathcal{T}}}(\Sigma^{<-m}\mathcal{M}),{}^{\perp_{\mathcal{T}}}(\Sigma^{>-m}\mathcal{M}))$ is also a $t$-structure with $\widetilde{\mathcal{T}}^{\leq m}\subseteq \mathcal{T}^{fd}$.
\end{remark}

The following can be proven by a similar argument to the one proving Theorem \ref{thm_gapinX}.
\begin{theorem}\label{thm_gapinXdual}
Let $a$ be an integer and $X\in\mathcal{T}$ be such that $\mathcal{T}(X,\Sigma^j M)=0$ for $a\leq j\leq a+n-2$. Then, the truncation triangle
\begin{align*}
    \xymatrix@C=3em{
    \widetilde{X}^{\leq -a}\ar[r]& X\ar[r]&\widetilde{X}^{\geq -a+1}\ar[r]& \Sigma \widetilde{X}^{\leq -a}
    }
\end{align*}
splits and $X\cong \widetilde{X}^{\leq -a}\oplus \widetilde{X}^{\geq -a+1}$.
\end{theorem}

We conclude this section with three lemmas setting us up for an application of Theorems \ref{thm_gapinX} and \ref{thm_gapinXdual}, see Corollary \ref{coro_n-1ct}. We first recall the definitions of precovers, also known as right approximations, and preenvelopes, also known as left approximations.
\begin{defn}\label{defn_cover_ffinite}
Let $\mathcal{T}$ be a triangulated category and $\mathcal{A}\subseteq\mathcal{T}$ be a full subcategory. An $\mathcal{A}$\textit{-precover} of $X\in\mathcal{T}$ is a morphism of the form $\alpha :A\rightarrow X$ with $A\in \mathcal{A}$ such that every morphism $\alpha':A'\rightarrow X$ with $A'\in \mathcal{A}$ factorizes as:
\begin{align*}
\xymatrix{
A'\ar[rr]^{\alpha'} \ar@{-->}[dr]_{\exists}& & X.\\
& A \ar[ru]_{\alpha} &
}
\end{align*}
%An $\mathcal{A}$\textit{-cover} of $X$ is a $\mathcal{A}$-precover of $X$ which is also a right minimal morphism.
The dual notion of precovers is \textit{preenvelopes}.

The subcategory $\mathcal{A}$ of $\mathcal{T}$ is called \textit{precovering} if every object in $\mathcal{T}$ has an $\mathcal{A}$-precover. Dually, it is called \textit{preenveloping} if every object in $\mathcal{T}$ has an $\mathcal{A}$-preenvelope. If $\mathcal{A}$ is both precovering and preenveloping, it is called \textit{functorially finite} in $\mathcal{T}$.
\end{defn}

The following is a consequence of \cite[Lemma 5.3]{SZ}, but we provide a proof for the benefit of the reader.
\begin{lemma}\label{lemma_a_star_b_ff}
Let $\mathcal{X}$ be a $k$-linear, Hom-finite, Krull-Schmidt, triangulated category and $\mathcal{A},\, \mathcal{B}\subset \mathcal{X}$ be full subcategories.
\begin{enumerate}[label=(\alph*)]
\item If $\mathcal{A},\, \mathcal{B}$ are precovering in $\mathcal{X}$,
then  $\mathcal{A}*\mathcal{B}$ is precovering in $\mathcal{X}$.
\item If $\mathcal{A},\, \mathcal{B}$ are preenveloping in $\mathcal{X}$,
then  $\mathcal{A}*\mathcal{B}$ is preenveloping in $\mathcal{X}$.
\end{enumerate}
\end{lemma}

\begin{proof}
We only prove (a), then (b) follows by a similar argument.
Assume $\mathcal{A},\, \mathcal{B}$ are precovering in $\mathcal{X}$ and let $X$ be an object in $\mathcal{X}$. Note that since $\mathcal{B}$ is precovering in $\mathcal{X}$, then so is $\Sigma^{-1}\mathcal{B}$. Then, there are two triangles in $\mathcal{X}$ of the form
\begin{align*}
     \xymatrix@C=3em @R=1em{
    X_1\ar[r]^-{x_1}& A\ar[r]^-{a}&X\ar[r]^-{x}& \Sigma X_1,\\
    X_2\ar[r]^-{x_2}& \Sigma^{-1}B\ar[r]^-{b}&X_1\ar[r]^-{x'}& \Sigma X_2,
    }
\end{align*}
where $a$ is an  $\mathcal{A}$-precover of $X$ and $b$ is a $\Sigma^{-1}\mathcal{B}$-precover of $X_1$. By the octahedral axiom, we have a commutative diagram of triangles in $\mathcal{X}$ of the form

\begin{align}\label{diagram_precover_c}
     \xymatrix@C=3em{
    X_2\ar[r]\ar[d]^{x^2}& 0\ar[r]\ar[d]&\Sigma X_2\ar@{=}[r]\ar[d]& \Sigma X_2\ar[d]^{\Sigma x^2}\\
    \Sigma^{-1} B\ar[r]^-{x_1\circ b}\ar[d]^{b}& A\ar[r]\ar@{=}[d]&C\ar[r]\ar[d]^c& B\ar[d]^{\Sigma b}\\
    X_1\ar[r]^-{x_1}\ar[d]^{x'}& A\ar[r]^-{a}\ar[d]&X\ar[r]^-{x}\ar[d]^y& \Sigma X_1\ar[d]^{\Sigma x'}\\
    \Sigma X_2\ar[r]& 0\ar[r]&\Sigma^2 X_2\ar@{=}[r]& \Sigma^2 X_2.
   }
\end{align}
We show that $c:C\rightarrow X$ is an $(\mathcal{A}*\mathcal{B})$-precover.
First, note that $C\in \mathcal{A}*\mathcal{B}$ since the second row in (\ref{diagram_precover_c}) is a triangle with  $A\in\mathcal{A}$ and $B\in\mathcal{B}$. Let $c': C'\rightarrow X$ be a morphism with $C'\in \mathcal{A}*\mathcal{B}$. Then $C'$ appears in a triangle in $\mathcal{X}$ of the form
\begin{align*}
    \xymatrix@C=3em{
    \Sigma^{-1} B'\ar[r]& A'\ar[r]^-{\overline{a}}&C'\ar[r]^-{\overline{c}}& B',
    }
\end{align*}
where $A'\in\mathcal{A}$ and $B'\in \mathcal{B}$.
Since $a:A\rightarrow X$ is an $\mathcal{A}$-precover, there exists a morphism of the form $a':A'\rightarrow A$ such that $a\circ a'=c'\circ \overline{a}$. Using this and $\Sigma x'\circ x=y$ by (\ref{diagram_precover_c}), there are morphisms of triangles of the form
\begin{align}\label{diagram_a_star_b}
     \xymatrix@C=3em{
    \Sigma^{-1} B'\ar[r]\ar@{-->}[d]^{h}& A'\ar[r]^{\overline{a}}\ar[d]^{a'}&C'\ar[r]^{\overline{c}}\ar[d]^{c'}& B'\ar@{-->}[d]^{\Sigma h}\\
    X_1\ar[r]^-{x_1}\ar[d]^{x'}& A\ar[r]^-{a}\ar@{-->}[d]&X\ar[r]^-{x}\ar@{=}[d]& \Sigma X_1\ar[d]^{\Sigma x'}\\
    \Sigma X_2\ar[r]& C\ar[r]^c&X\ar[r]^y& \Sigma^2 X_2.
   }
\end{align}
Since $b:\Sigma^{-1}B\rightarrow X_1$ is a $\Sigma^{-1}\mathcal{B}$-precover and $\Sigma^{-1}B'\in\Sigma^{-1}\mathcal{B}$, there is a morphism $\overline{b}:\Sigma^{-1}B'\rightarrow\Sigma^{-1}B$ such that $b\circ \overline{b}=h$.
Hence
\begin{align*}
    x'\circ h=x'\circ b\circ \overline{b}=0\circ \overline{b}=0,
\end{align*}
where $x'\circ b=0$ as two consecutive morphisms in a triangle compose to zero.
By commutativity of (\ref{diagram_a_star_b}), we then have
\begin{align*}
    0=\Sigma (x'\circ h)\circ \overline{c}=\Sigma x'\circ \Sigma h\circ \overline{c}=y\circ c'.
\end{align*}
Hence there is a morphism $s:C'\rightarrow C$ such that $c\circ s=c'$.
So $c:C\rightarrow X$ is an $(\mathcal{A}*\mathcal{B})$-precover of $X$ and $\mathcal{A}*\mathcal{B}$ is precovering in $\mathcal{X}$.
\end{proof}

\begin{lemma}\label{lemma_W_torsionpair}
Let $p$ be an integer. Then the subcategory
\begin{align*}
    \mathcal{W}_p:=\Sigma^p\mathcal{M}*\Sigma^{p+1}\mathcal{M}*\cdots*\Sigma^{p+n-2}\mathcal{M}\subseteq \mathcal{T}
\end{align*}
is functorially finite and closed under extensions and so  $(\mathcal{W}_p,\mathcal{W}_p^{\perp_\mathcal{T}})$ and $(^{\perp_{\mathcal{T}}}\mathcal{W}_p,\mathcal{W}_p)$ are torsion pairs.
\end{lemma}

\begin{proof}
Recall that $\mathcal{M}=\text{add}\,(M)$ by assumption, that is $\mathcal{M}$ has finitely many indecomposables. Then, $\mathcal{M}$ is both precovering and preenveloping in $\mathcal{T}$ and so is $\Sigma^i\mathcal{M}$ for any integer $i$. Then, using Lemma \ref{lemma_a_star_b_ff} repeatedly, for any integer $p$ we have that $\mathcal{W}_p$ is functorially finite.

Now let $i\geq j$ be integers and $X\in \Sigma^i\mathcal{M}*\Sigma^j\mathcal{M}$. Then, there is a triangle in $\mathcal{T}$ of the form
\begin{align*}
    \Sigma^i M'\rightarrow X\rightarrow\Sigma^j M''\xrightarrow{\delta} \Sigma^{i+1} M',
\end{align*}
with $M',\,M''\in\mathcal{M}$. Since $\mathcal{T}(\mathcal{M},\Sigma^{\geq 1} \mathcal{M})=0$ and $i\geq j$, we have that $\delta=0$ and so $X\cong \Sigma^i M'\oplus\Sigma^j M''$. Hence if $i\geq j$, then
\begin{align}\label{diagram_a_star_or_oplus_b}
    \Sigma^i\mathcal{M}*\Sigma^j\mathcal{M}=\Sigma^i\mathcal{M}\oplus\Sigma^j\mathcal{M}.
\end{align}
Since $n\geq 3$, we can use (\ref{diagram_a_star_or_oplus_b}) repeatedly to see that
\begin{align*}
    \mathcal{W}_p*\mathcal{W}_p=\Sigma^p\mathcal{M}*\Sigma^{p+1}\mathcal{M}*\cdots*\Sigma^{p+n-2}\mathcal{M}*\Sigma^p\mathcal{M}*\Sigma^{p+1}\mathcal{M}*\cdots*\Sigma^{p+n-2}\mathcal{M}=\mathcal{W}_p
\end{align*}
and so $\mathcal{W}_p$ is closed under extensions.
Then, the fact that $(\mathcal{W}_p,\mathcal{W}_p^{\perp_\mathcal{T}})$ and $(^{\perp_{\mathcal{T}}}\mathcal{W}_p,\mathcal{W}_p)$ are torsion pairs follows from \cite[Proposition 2.3(1)]{IY} and its dual.
\end{proof}

The following is a well-known fact which has been used in various arguments in previous papers, such as for example in \cite[proof of Theorem 5.8(c)]{IYa}.
\begin{lemma}\label{lemma_dct_eq_defn}
Let $d\geq 2$ be an integer. A  subcategory $\mathcal{C}=\text{add}(C)\subseteq \mathcal{T}$ is $d$-cluster tilting if and only if  $\mathcal{T}(C,\Sigma^i C)=0$ for $1\leq i\leq d-1$ and 
\begin{align*}
    \mathcal{T}\cong \mathcal{C}*\Sigma\mathcal{C}*\cdots *\Sigma^{d-1}\mathcal{C}.
\end{align*}
\end{lemma}

\section{Applications of our Theorems}\label{section_coro}
In this section, we use our theorems to reprove some important properties of the triangulated category $\mathcal{T}/\mathcal{T}^{fd}$ and of its object $M$.
\begin{remark}\label{remark_section5}
Corollaries \ref{coro_CY}, \ref{coro_n-1ct} and \ref{coro_part3} correspond to the three parts of \cite[Theorem 2.2]{GL}, which is a higher version of \cite[Theorem 2.1]{AC}. Note that both Amiot and Guo work in the more specific setup presented in Remark \ref{remark_AG_setup}.
%They fix a differential graded (dg) $k$-algebra $A$ with some properties, see \cite[(1)-(3) and (4') in Section 1]{GL}, and take $\mathcal{T}=\text{per}\, A$ to be the perfect derived category of $A$ and $\mathcal{T}^{fd}=\mathcal{D}^b A$ to be the full subcategory of the derived category of dg $A$-modules whose objects are the objects with finite-dimensional total homology. Then this setup satisfies Setup \ref{setup}, with $M=A$.
\end{remark}

Theorem \ref{thm_directstab} has the following consequence, corresponding to \cite[Theorem 2.2(1)]{GL}.

\begin{corollary}\label{coro_CY}
The category $\mathcal{T}/\mathcal{T}^{fd}$ is Hom-finite and $(n-1)$-Calabi-Yau.
\end{corollary}
\begin{proof}
Let $X$ and $Y$ be objects in $\mathcal{T}$. By Theorem \ref{thm_directstab}, for $p\gg0$ we have that
\begin{align*}
  \mathcal{T}/\mathcal{T}^{fd}(X, Y) \cong  \mathcal{T}(X^{\leq -p}, Y)
\end{align*}
and $\mathcal{T}$ is Hom-finite by assumption, see (CY1). Hence $\mathcal{T}/\mathcal{T}^{fd}$ is Hom-finite.

Now let $p$ and $q$ be integers. Applying the functor $D\mathcal{T}(Y^{\leq -q},-)$ to a shift of the triangle
\begin{align*}
    \xymatrix@C=3em{
    X^{\leq -p}\ar[r]& X\ar[r]&X^{\geq -p+1}\ar[r]& \Sigma X^{\leq -p},
    }
\end{align*}
we obtain the exact sequence
\begin{align}\label{diagram_prop2}
    \xymatrix@C=0.8em{
    D\mathcal{T}(Y^{\leq -q},\Sigma^n X^{\leq -p})\ar[r]&D\mathcal{T}(Y^{\leq -q},\Sigma^{n-1}X^{\geq -p+1})\ar[r]&D\mathcal{T}(Y^{\leq -q},\Sigma^{n-1}X)\ar[r]&D\mathcal{T}(Y^{\leq -q}, \Sigma^{n-1}X^{\leq -p}).
    }
\end{align}
Taking the direct limit of (\ref{diagram_prop2}) with respect to $p$, we get the exact sequence
\begin{align}\label{diagram_prop3}
    \xymatrix@C=2em{
    0\ar[r]&\underset{p}{\varinjlim}\,D\mathcal{T}(Y^{\leq -q},\Sigma^{n-1}X^{\geq -p+1})\ar[r]^-{\sim}&D\mathcal{T}(Y^{\leq -q},\Sigma^{n-1}X)\ar[r]&0,
    }
\end{align}
where
\begin{align*}
    \underset{p}{\varinjlim}\,D\mathcal{T}(Y^{\leq -q},\Sigma^n X^{\leq -p})=0 \text{ and } \underset{p}{\varinjlim}\,D\mathcal{T}(Y^{\leq -q},\Sigma^{n-1} X^{\leq -p})=0
\end{align*}
by Lemma \ref{lemma_pbig_Y_thick}(b) and because $q$ is fixed. 
For $q\gg0$, we then have
\begin{align}\label{diagram_prop5}
    \underset{p}{\varinjlim}\,D\mathcal{T}(Y^{\leq -q},\Sigma^{n-1}X^{\geq -p+1})\cong
    D\mathcal{T}(Y^{\leq -q},\Sigma^{n-1}X)\cong
    D\mathcal{T}/\mathcal{T}^{fd}(Y,\Sigma^{n-1}X),
\end{align}
where the first isomorphism is obtained by (\ref{diagram_prop3}) and the second one follows by Theorem \ref{thm_directstab}(a), as $q\gg0$.
Then, for $q\gg0$, we have
%Note that, since $X\in\text{thick}(M)$, for $q\gg0$ we have that $\mathcal{T}(X, Y^{\leq -q})=0$ and $\mathcal{T}(\Sigma^{-1}X, Y^{\leq -q})=0$. Then, using exactness of (\ref{diagram_prop4}), the fact that $\Sigma^{-1}X^{\geq -p+1}\in\mathcal{T}^{fd}$  and  (\ref{diagram_prop5}), for $q\gg0$ we have
\begin{align*}
    \mathcal{T}/\mathcal{T}^{fd}(X,Y)\cong
    \underset{p}{\varinjlim}\,\mathcal{T}(\Sigma^{-1}X^{\geq -p+1},Y^{\leq -q})\cong
    \underset{p}{\varinjlim}\,D\mathcal{T}(Y^{\leq -q},\Sigma^{n-1}X^{\geq -p+1})\cong
    D\mathcal{T}/\mathcal{T}^{fd}(Y, \Sigma^{n-1}X).
\end{align*}
In the above, the first isomorphism follows by combining Lemma \ref{lemma_iso_quotient_cat} and  Theorem \ref{thm_directstab}(a), the second one by applying (CY2) from Definition \ref{defn_nCYtriple} to $\Sigma^{-1}X^{\geq -p+1}\in\mathcal{T}^{fd}$ and $Y^{\leq -q}$ and the third one follows by (\ref{diagram_prop5}). 
Hence $\mathcal{T}/\mathcal{T}^{fd}$ is $(n-1)$-Calabi-Yau.
\end{proof}

Before presenting the next result, we recall the definition of left and right minimal morphisms, which will be used in the proof.

\begin{defn}
A morphism $\alpha:A\rightarrow B$ in a category $\mathcal{A}$ is left minimal if each morphism $\eta: B\rightarrow B$ which satisfies $\eta\circ \alpha=\alpha$ is an isomorphism. Dually, $\alpha$ is right minimal if each morphism $\varphi:A\rightarrow A$ which satifies $\alpha\circ\varphi=\alpha$ is an isomorphism.
\end{defn}

As a consequence of Theorems \ref{thm_gapinX} and \ref{thm_gapinXdual}, we obtain the following, which corresponds to \cite[Theorem 2.2(2)]{GL}.

\begin{corollary}\label{coro_n-1ct}
The object $M$ is $(n-1)$-cluster tilting in $\mathcal{T}/\mathcal{T}^{fd}$.
\end{corollary}

\begin{proof}
By Corollary \ref{coro_MandSigma_zero}, we have that $\mathcal{T}/\mathcal{T}^{fd}(M,\Sigma^j M)=0$ for $j=1,\dots, n-2$. Hence, by Lemma \ref{lemma_dct_eq_defn}, in order to prove that $M$ is $(n-1)$-cluster tilting in $\mathcal{T}/\mathcal{T}^{fd}$, it is enough to prove that
\begin{align*}
    \mathcal{W}_0=\mathcal{M}*\Sigma\mathcal{M}*\cdots*\Sigma^{n-2}\mathcal{M}=\mathcal{T}/\mathcal{T}^{fd}.
\end{align*}

Let $X$ be an object in $\mathcal{T}$ and $p$ be an integer. By Lemma \ref{lemma_W_torsionpair}, we have that $(\mathcal{W}_p,\mathcal{W}_p^{\perp_\mathcal{T}})$ is a torsion pair and so $\mathcal{T}=\mathcal{W}_p*\mathcal{W}_p^{\perp_\mathcal{T}}$. Hence there exists a triangle in $\mathcal{T}$ of the form
\begin{align*}
    W\xrightarrow{f} X\xrightarrow{g} V\rightarrow \Sigma W,
\end{align*}
where $W\in\mathcal{W}_p$, $V\in\mathcal{W}_p^{\perp_{\mathcal{T}}}$ and without loss of generality we may assume that $f$ is a right minimal morphism. Then, by \cite[Lemma 2.5]{KH}, we have that $g$ is a left minimal morphism.
Since $V\in \mathcal{W}_p^{\perp_{\mathcal{T}}}$, we have that 
\begin{align*}
    \mathcal{T}(M,\Sigma^j V)=0
\end{align*}
for $-p-n+2\leq j\leq -p$. Then, by Theorem \ref{thm_gapinX} we have that
\begin{align*}
    V\cong V^{\leq -p-n+1}\oplus V^{\geq -p-n+2}.
\end{align*}
By Lemma \ref{lemma_pbig_Y_thick}(b), for $p\gg0$ we have that $\mathcal{T}(X,V^{\leq -p-n+1})=0$ and as $g$ is left minimal, it follows that $V^{\leq -p-n+1}=0$ and
\begin{align*}
    V\cong_{\mathcal{T}/\mathcal{T}^{fd}} 0 \text{ and } X\cong_{\mathcal{T}/\mathcal{T}^{fd}} W\in\mathcal{W}_p.
\end{align*}
Hence $X\in\mathcal{W}_p$ for some $p\gg0$.

We now show that if $a\geq 0$ is an integer, then $\mathcal{W}_{a+1}\subseteq_{\mathcal{T}/\mathcal{T}^{fd}} \mathcal{W}_{a}$. Let $Y\in\mathcal{W}_{a+1}$. Since $({}^{\perp_{\mathcal{T}}}\mathcal{W}_a,\mathcal{W}_a)$ is a torsion pair by Lemma \ref{lemma_W_torsionpair}, there is a triangle of the form
\begin{align*}
    U\xrightarrow{u} Y\xrightarrow{y} W'\rightarrow \Sigma U,
\end{align*}
where $U\in {}^{\perp_\mathcal{T}}\mathcal{W}_a$, $W'\in\mathcal{W}_a$ and without loss of generality we may assume that $y$ is left minimal, so that $u$ is right minimal by \cite[Lemma 2.5]{KH}. Note that since $U\in {}^{\perp_\mathcal{T}}\mathcal{W}_a$, we have that
\begin{align*}
    \mathcal{T}(U,\Sigma^j M)=0
\end{align*}
for $a\leq j\leq a+n-2$. Then, by Theorem \ref{thm_gapinXdual}, we have that
\begin{align*}
    U\cong \widetilde{U}^{\leq -a}\oplus \widetilde{U}^{\geq -a+1}.
\end{align*}
Since $\widetilde{U}^{\geq -a+1}\in\widetilde{\mathcal{T}}^{\geq -a+1}$ and $Y\in\mathcal{W}_{a+1}$, we have that
\begin{align*}
    \mathcal{T}(\widetilde{U}^{\geq -a+1},\Sigma^{\geq a} M)=0 \text{ and } \mathcal{T}(\widetilde{U}^{\geq -a+1}, Y)=0.
\end{align*}
As $u$ is right minimal, it follows that $\widetilde{U}^{\geq -a+1}=0$ and $U\cong_{\mathcal{T}/\mathcal{T}^{fd}} 0$. Hence
\begin{align*}
    Y\cong_{\mathcal{T}/\mathcal{T}^fd} W'\in\mathcal{W}_{a}
\end{align*}
and so $\mathcal{W}_{a+1}\subseteq_{\mathcal{T}/\mathcal{T}^{fd}}\mathcal{W}_a$ as we wanted to show.

Since the above is true for arbitrary $a\geq 0$, we have
\begin{align*}
    \cdots\subseteq_{\mathcal{T}/\mathcal{T}^{fd}}\mathcal{W}_p\subseteq_{\mathcal{T}/\mathcal{T}^{fd}}\cdots \subseteq_{\mathcal{T}/\mathcal{T}^{fd}}\mathcal{W}_2\subseteq_{\mathcal{T}/\mathcal{T}^{fd}}\mathcal{W}_1\subseteq_{\mathcal{T}/\mathcal{T}^{fd}}\mathcal{W}_0.
\end{align*}
So we conclude that
\begin{align*}
    \mathcal{W}_0=\mathcal{M}*\Sigma\mathcal{M}*\cdots*\Sigma^{n-2}\mathcal{M}=\mathcal{T}/\mathcal{T}^{fd}
\end{align*}
and $M$ is $(n-1)$-cluster tilting in $\mathcal{T}/\mathcal{T}^{fd}$.
\end{proof}

As a consequence to Corollary \ref{coro_valueforp}, we obtain the following, which corresponds to \cite[Theorem 2.2(3)]{GL}.
\begin{corollary}\label{coro_part3}
We have that $\mathcal{T}/\mathcal{T}^{fd}(M,M)\cong\mathcal{T}(M,M)$.
\end{corollary}
\begin{proof}
We have that
\begin{align*}
    \mathcal{T}/\mathcal{T}^{fd}(M,M)\cong\mathcal{T}(M^{\leq n-2},M)\cong\mathcal{T}(M,M),
\end{align*}
where the first isomorphism holds by Corollary \ref{coro_valueforp} with $X=M$ and $j=0$ and the second one, since $n-2\geq 0$, holds by Lemma \ref{lemma_M_iso_M_leq}.
\end{proof}

\begin{remark}
At first glance, the statement of Corollary \ref{coro_part3} is different from \cite[Theorem 2.2(3)]{GL}. However, in Guo's setup, see Remarks \ref{remark_AG_setup} and \ref{remark_section5}, we have that
\begin{align*}
    \Hom_{\text{per}\, A} (A,-)\cong H^0(-).
\end{align*}
Hence \cite[Theorem 2.2(3)]{GL} is equivalent to
\begin{align*}
    \Hom_{\text{per}\, A / \mathcal{D}^{b} A} (A,A)\cong \Hom_{\text{per}\, A} (A,A),
\end{align*}
which is a special case of Corollary \ref{coro_part3}.
\end{remark}

\end{document}